\title{$N$-sociable Heronian triangles}
\author{Nart Shalqini\thanks{
                  Franklin and Marshall College}}

\documentclass{article}

\usepackage{graphicx}
\usepackage{amsmath}
\usepackage{amssymb}

\newtheorem{theorem}{Theorem}
\newtheorem{lemma}{Lemma}

\newtheorem{definition}{Definition}
\newenvironment{proof}{{\sc Proof:}}{~\hfill QED}

\begin{document}
\newpage
\maketitle

\begin{abstract}
A \emph{Heronian} triangle is a triangle that has integer side lengths and integer area. Praton and Shalqini \cite{Praton} define amicable Heronian triangles to be two Heronian triangles where the area of one equals the perimeter of the other, and vice versa; analogous to the concept of amicable numbers. In this paper, we generalize this notion to $n$\emph{-sociable} Heronian triangles and characterize them for each $n$.
\end{abstract}

\section*{Introduction}
A \emph{Heronian} (or \emph{Heron}) triangle is a triangle that has integer side lengths and integer area. The name comes from the ancient Greek mathematician Heron of Alexandria, who first provided the formula for the area of a triangle solely in terms of its side lengths,
\begin{align}
A= \sqrt{s(s-a)(s-b)(s-c)},   
\end{align}
where $a, b, c$ are the side lengths of the triangle and $s$ is the semi-perimeter i.e., $s=(a+b+c)/2$. We define a Heronian triangle to be \emph{equable} if its area equals its perimeter. It turns out that there are only five equable triangles \cite{nth}, namely the triangles with side lengths $ (5, 12, 13), (6, 8, 10), (6, 25, 29), (7, 15, 20)$, and $(9, 10, 17)$. The last equable triangle will be of particular interest to us. 

Equable triangles are sometimes called \emph{superhero} triangles as featured in \cite{numberphile}. They are analogous to the concept of perfect numbers; numbers that are equal to the sum of their proper divisors (aliquot sum), for example $6$ and $28$. Praton and Shalqini \cite{Praton} define an \emph{amicable} pair of triangles in similar fashion to \emph{amicable} numbers, which are numbers whose aliquot sum equal each other. The smallest pair of amicable numbers is $(220, 284)$. Although there are finitely many equable and amicable Heronian triangles \cite{Praton}, the case for perfect and amicable integers remains an open question.

The next natural thing that comes after perfect and amicable numbers are \emph{sociable} numbers. An $n$\emph{-sociable} cycle is a set of numbers whose aliquot sums form a cyclic sequence. In other words, a set of integers $a_1, a_2, \dots, a_n$ is an $n$\emph{-sociable} cycle if $s(a_1)=a_2, s(a_2)=a_3, \dots, s(a_n)=a_1$ where $s(a_i)$ represents the aliquot sum of $a_i$. Note that the amicable numbers are a special case that occurs when $n=2$. So far, research has largely focused on perfect and amicable numbers, and not much is known about sociable numbers. In this paper we define $n$\emph{-sociable} Heronian triangles analogously to $n$\emph{-sociable} numbers. Surprisingly, the complete characterization of $n$\emph{-sociable} Heronian triangles turns out be possible by only using elementary techniques.

\section*{Definition and proofs}

\begin{definition}
Suppose we have $n$ Heronian triangles $H_1, H_2, \dots, H_n$ such that at least one of them is not equable. We say that this set is an  $n$\emph{-sociable} Heronian cycle if $A_1 = p_2, A_2 = p_3, \dots, A_n=p_1$ where $A_i$ and $p_i$ are the area and perimeter of $H_i$, respectively. Moreover, for each $k$ such that $1 \leq k \leq n$ the Heronian cycles  $H_1, H_2, \dots, H_n$ and $H_k, H_{k+1}, \dots, H_{n-k}$ are equivalent.
\end{definition}

\begin{figure}[hbt]
    \centering
    \includegraphics[width=4.5in]{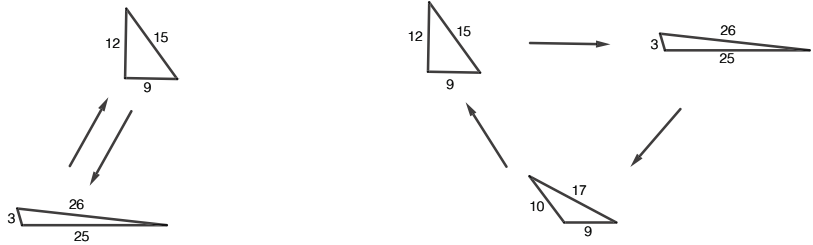}
    \caption{Left: the unique amicable pair. Right: a $3$\emph{-sociable} cycle. }
    \label{exp}
\end{figure}

The amicable pair defined in \cite{Praton} is equivalent to a $2$\emph{-sociable} Heronian cycle. Furthermore, \cite{Praton} shows that there is a unique such pair of triangles; namely, the triangles $H_1: (9, 12, 15)$ and $H_2 :(3, 25, 26)$ where $A_1=54=p_2$ and $A_2=36=p_1$. In this paper, we give a complete characterization of $n$\emph{-sociable} Heronian cycles for all $n$. We show that all $n$\emph{-sociable} cycles are only composed of a combination of our amicable pairs with the equable triangle $(9, 10, 17)$. This in turn implies that there are finitely many such cycles for each $n$.  We begin by establishing a series of lemmas and theorems necessary for our proof.
     
\begin{lemma}\label{lm1}
     Each $n$\emph{-sociable} Heronian cycle contains a triangle whose perimeter is strictly greater than its area.
\end{lemma}
\begin{proof}
    Let $H_1, H_2, \dots, H_n$ be a cycle of $n$\emph{-sociable} Heronian triangles. By definition, there exists some $H_i$ such that $A_i \neq p_i$. Without loss of generality, we can assume that $i=1$. If we have some other triangle $H_j$ ($j \neq 1$) such that $p_j > A_j$ then we are done. If not, then $A_1 = p_2 \leq A_2 = p_3 \leq A_3 = \dots \leq A_{n} = p_{1} $. This implies that $A_1 < p_1$ since $H_1$ is not equable, proving our claim.
\end{proof}

\begin{lemma}\label{lm2}
     Let $H$ be a Heronian triangle in an $n$\emph{-sociable} Heronian cycle with perimeter $p$ and area $A$. Then $A^2/p$ is an integer.
\end{lemma}
\begin{proof}
    First we show that for any Heronian triangle, the semi-perimeter is an integer. Let $a,b,c$ be the side lengths of $H$. Suppose $s=(a+b+c)/2$. We let $x=s-a, y=s-b$ and $z=s-c$. We can further suppose $z \geq y \geq x$. Note that $s=x+y+z$, and that $x+y=c$, $x+z=b$ and $y+z=a$. By Heron's formula, $\sqrt{sxyz}= A$ which implies $sxyz=A^2 \in \mathbb{N}$. Note that $s$ must be an integer. Otherwise $s$ would be a half-integer, meaning that $x, y, z$ are all half integers. This implies that $sxyz$ equals an odd integer divided by $16$; an impossibility given the area is an integer. 
    
    Assuming $H$ is in an $n$\emph{-cycle}, we know that $A=p'$ where $p'$ is the perimeter of the proceeding Heronian triangle. Since $p'=2s'$ and $s'$ is an integer, then $A$ is even. This implies at least one of $x,y,z$ is even. Therefore
    \begin{align}
      \frac{A^2}{p} = \frac{sxyz}{2s} = \frac{xyz}{2} \in \mathbb{N}.
    \end{align}
\end{proof}

\begin{theorem}\label{thm1}
 Let $H$ be a triangle like the one in Lemma \ref{lm2}. Then, $z$ divides $2^{2^n}(x+y)^{2^n-1}$.
\end{theorem}
\begin{proof}
Suppose $H_1, H_2, \dots, H_n$ is an $n$\emph{-sociable} Heronian cycle. Without loss of generality, assume that $H=H_1$. From Lemma \ref{lm2} we know that for each $i$, $A_i^2 = c_i p_i$ for some $c_i \in \mathbb{N}$. Hence
\begin{align*}
 p_1^{2^n} &=\left( A^2_{n}\right)^{2^{n-1}}  = \left(c_{n} p_{n} \right)^{2^{n-1}} = c^{2^n-1}_{n}\left(p^2_{n}\right)^{2^{n-2}} \\
 &=c_n^{2^n-1}\left(A^2_{n-1}\right)^{2^{n-2}} = c^{2^n-1}\left(c_{n-1}p_{n-1} \right)^{2^{n-2}}= \dots \\
 &= c^{2^{n-1}}_{n}c^{2^{n-2}}_{n-1}  \cdots   c_2^2 A_1^2.
\end{align*}
Thus $p_1^{2^{n}}/A_1^2 \in \mathbb{N}$, and we have
\begin{align}\label{inn}
    \frac{p_1^{2^{n}}}{A_1^2}= \frac{2^{2^{n}}s^{2^n}}{sxyz} = 2^{2^n} \frac{s^{2^n-1}}{xyz} \in \mathbb{N} & \Rightarrow 2^{2^n} \frac{s^{2^n-1}}{z} \in \mathbb{N}.
\end{align}  
If $c=x+y$, then by (\ref{inn}) the Binomial Theorem yields
\begin{align*}
2^{2^n}\frac{(c+z)^{2^n-1}}{z} \in \mathbb{N} & \Rightarrow 2^{2^n} \frac{\sum_{k=0}^{{2^n-1}}\binom{2^n-1}{k}z^{k}c^{2^n-1-k}}{z} \in \mathbb{N} \\ &\Rightarrow 2^{2^n} \frac{\sum_{k=1}^{2^n-1} \binom{2^n-1}{k} z^k c^{2^n-1-k}}{z} + 2^{2^n}\frac{c^{2^n-1}}{z} \in \mathbb{N} \\
&\Rightarrow 2^{2^n} \sum_{k=1}^{2^n-1}\binom{2^n-1}{k} z^{k-1}c^{2^n-1-k} + 2^{2^n}\frac{c^{2^n-1}}{z} \in \mathbb{N}.
\end{align*}
The first summand is an integer, which implies that the second summand must also be an integer. This is possible if and only if $z \ | \ 2^{2^n} c^{2^n-1}$, completing our proof.
\end{proof}

\begin{lemma}\label{lm3}
Assume that $H$ is a Heronian triangle such that its perimeter $p$ is greater than its area $A$. Then $x \leq 3$ and $y \leq 9$.
\end{lemma}
\begin{proof}
This follows from Lemma 3 in \cite{Praton}.
\end{proof}

Consequently, Theorem \ref{thm1} implies that there are finitely many $n$-sociable Heronian triangles whose perimeter is greater than their area, since $z$ has to divide finitely many integers; namely the integers $2^{2^n}(x+y)^{2^n-1}$ where $x\leq 3$ and $y \leq 9$. The next theorem shows that there are at most four of them.

\begin{theorem}\label{thm2}
Let $H$ be a Heronian triangle in an $n$\emph{-sociable} cycle such that its perimeter is greater than its area. Then $H$ must be one of the following triangles: $(3,4,5), (5,5, 8), (3, 25, 26)$ and $(3, 865, 866)$
\end{theorem}
\begin{proof}
 First, we should note that for $H$, we have
 \begin{align}\label{pa}
   p>A \Rightarrow 2s >A \Rightarrow 4s^2 > A^2 \Rightarrow 4s^2 > sxyz \Rightarrow 4(x+y+z) > xyz. 
 \end{align}
 From Lemma \ref{lm3} it follows that $x=1$, $x=2$, or $x=3$ so we treat each of these cases separately. We assume, as previously, that $x \leq y \leq z$.

 We begin with $x=3$. If $y \geq 4$, then by (\ref{pa}) we have $4(z+z+z) \geq 4(x+y+z) > xyz = 3yz \geq 12z $. Hence $12z > 12z$; a contradiction. So $y=3$, and $z=3$ or $z=4$ since $4(3+3+z) >xyz = 9z$. Neither of these values yield an integer area.
 
 Next, we consider $x=2$. For $y > 5$, we have $xy$ is at least $12$, so $12z > 12z$ and a contradiction. If $y=5$, then $4(2+5+z)>10z$, so $28>6z$ and $z<5$, a contradiction arises again. If $y=2$ then $sxyz=A^2$ which implies $(4+z)4z=A^2$. Hence
$
16 = (2z+ 4 - A)(2z+4+A).
$ After checking all possible factorizations of $16$, we conclude that there are no positive integers solutions for $z$. Supposing $y=3$ we get $z<10$. Then $z=4, 5, 8$ since $z$ divides $2^{2^n}(x+y)^{2^n-1} = 2^{2^n}(5)^{2^n-1}$ by Theorem \ref{thm1}. None of these values yield an integer area. Similarly if $y=4$ then $z<6$, which means $z=4,5$. Neither of these values work either.  

Suppose that $x=1$. By Lemma \ref{lm3}, $y \leq 9$. From (\ref{pa}), we get an upper bound on $z$ such that $z < 4 (y+1)/(y-4)$ when $y>4$. If $y=9$, then $z<9=y$, a contradiction. If $y=8$, then $z=8$ which yields an irrational area. If $y=7$, then $7 \leq z \leq 10$, none of which yield a rational area. If $y=6$ then $ 6 \leq z \leq 13$. The only values of $z$ that divide $2^{2^n} \cdot 7^{2^n-1}$ are $7$ and $8$, but both of them yield an irrational area. Finally if $y=5$ then $ 5 \leq z \leq 23$. The only values of $z$ that divide $2^{2^{n+1}-1}3^{2^n-1}$ are $6, 8, 9, 12, 16$ and $18$. None of these values work either.

We need to take more care for $y \leq 4$. Suppose that $y=4$. Again, we want $4z(5+z)=A^2$ for $A,z \in \mathbb{N}$. As before, we get $25= (2z+5-A)(2z+5+A)$. The only solution is $z=4$ and $A=12$, so $x=1, y=4, z=4$ is a candidate for $H$. If $y=1$ then $z(z+2)=A^2$ which is not possible since $A^2 = z(z+2) = (z+1)^2-1$ and no two perfect squares differ by one except the trivial case $z=0$ and $A=0$. 
    
Next, supposing $x=1$ and $y=3$ we get $3z(z+4)=A^2$. Also note that $z\ |\ 2^{2^n}(1+3)^{2^n-1} \Rightarrow z\ |\ 2^{2^n}4^{2^n-1}$. Therefore, $z=2^a$ for $a \in \mathbb{N}$. Clearly, $z \geq 3$, so $a \geq 2$. Hence $3\cdot 2^a (2^a+4) =3 \cdot 2^{a+2}(1+2^{a-2}) = A^2$. Note that $a$ cannot be odd; otherwise, the quantity $2^{a+2}(1+2^{a-2})$ would have an odd number of twos and hence wouldn't be a perfect square. If $a$ is even, then $1+2^{a-2} = 3c^2$ where $c$ is an integer. Thus $3c^2 = 1+2^{a-2}= 1+4 \cdot 2^{a-4}$ for $a \geq 4$ or $3c^2=2$ for $a=2$. Considering the equations modulo $4$, we see each case is impossible. Hence $y=3$ also leads to a contradiction.
    
Finally, suppose $x=1$ and $y=2$. Then $z$ has to divide $2^{2^n}(1+2)^{2^n-1} = 2^{2^n}3^{2^n-1}$, so $z=2^a3^b$ for some $a,b \in \mathbb{N}$. Furthermore, $A^2= 2z(z+3) = 2\cdot2^a3^b(2^a3^b+3) = 2^{a+1}3^b(2^a3^b+3)$. If $b=0$ we get $2^{a+1}(2^a+3)=A^2$, which implies that $a$ is odd and $2^a+3=c^2$. We can further assume that $a \geq 2$ ($z=2$ does not yield a rational area). The previous expression becomes $4\cdot 2^{a-2}+3=c^2$ which is impossible modulo $4$, so we can suppose that $b \geq 1$ and get $A^2 = 2^{a+1}3^{b+1}(2^a3^{b-1}+1)$.  Now consider $b \geq 2$. Then, $2^a3^{b-1}+1=c^2$ where $a,b$ are odd. Hence $2^a3^{b-1}=(c-1)(c+1)$ which implies $2^p3^q-2^r3^s=2$ where $p+r=a$ and $q+s=b-1$. So $2^{p-1}3^q - 2^{r-1}3^s=1$ where $p-1\geq 0$ and $r-1 \geq 0$; otherwise we have an even number equal an odd number. In turn, this means that $r-1=0$ and $q=0$ ($2^{p-1} -3^s=1$) or $p-1=0$ and $s=0$ ($3^q-2^{r-1}=1$). Otherwise the difference is either even or a multiple of three. Since $b \geq 2$, either $s \geq 1$ or $q \geq 2$. The solution to such equations is well-known and first proven by the medieval mathematician Gersonides \cite{gersonides} using arithmetic modulo $8$. The only solution satisfying our conditions is $q=2$, $r-1=3$ or $p-1=2, s=1$. The latter yields $a=4$ and $b=2$ which is not possible because $a$ and $b$ are both odd. The former yields $a=5$ and $b=2$ so $z=864$. Assuming $b=1$ we have $A^2=2^{a+1}3^2(2^a+1)$. This implies that either $a=0$ which yields $z=3$ or that $a$ is odd and $2^a+1=c^2$ for some $c \in \mathbb{N}$. This in turn implies that $a=3$ and $z=24$ using a similar argument as before.
    
We have just shown that the only possible values of $x, y, z$ are $x=1, y=4, z=4$, or $x=1, y=2, z=3$, or  $x=1, y=2, z=24$, or $x=1, y=2, z=864$. These values yields the following triangles: $ (5,5,8), (3,4,5), (3, 25, 26)$ and $(3, 865, 866)$.

\end{proof}

From Lemma \ref{lm1} we know that each $n$\emph{-sociable} Heronian cycle must contain a triangle whose perimeter is greater than its area. Hence, theorem \ref{thm2} provides us with a possibility of characterizing all $n$\emph{-sociable} Heronian cycles. The next step is to investigate each of these four triangles and see whether a sociable cycle contains them, and if so, what kind of other triangles must it contain. It turns out that three of the four cannot be part of this cycle, except for the one that belongs to our amicable pair. 

\begin{theorem}\label{thm3}
There is a unique triangle $H$ that belongs to a $n$-\emph{sociable} cycle such that its perimeter is greater than its area; namely, the triangle with side lengths $3$, $25$ and $26$. Moreover, each $n$-\emph{sociable} cycle is composed only of the amicable pairs and the triangle $(9, 10, 17)$. 
\end{theorem}
\begin{proof}
 Suppose $H$ belongs to an $n$-\emph{sociable} cycle such that its perimeter is greater than its area, and assume it is the first member of our cycle i.e., $H=H_1$. Then there are finitely possibilities for $H_1$ as listed in Theorem \ref{thm2}. Let $H_1$ be the triangle with side lengths $(5,5,8)$. This triangle has area $12$ and perimeter $18$. The semi-perimeter of the successor triangle $H_2$, if it exists, must be $6$. So for the successor triangle $x=1,y=1, z=4$ or $x=1,y=2,z=3$, or $x=2, y=2, z=2$. The first and third do not produce integer areas and the second triangle is our other candidate; namely, the triangle with side lengths $(3, 4, 5)$. The successor of this triangle must have perimeter $6$, and there is no Heronian triangle with such perimeter by a similar argument as above. Note that these arguments are the same as the ones used in \cite{Praton}. The triangle with side lengths $(3, 865, 866)$ needs to be considered with more care. Since this triangle has perimeter $1734$, then the preceding triangle $H_n$ must have area $1734= 2 \cdot 3 \cdot 17^2$. So, for the preceding triangle, if it exists, $2^2\cdot3^2\cdot17^4 = sxyz$. The fact that $s=x+y+z$ helps cut down the cases considerably and we find that the only possible values are $z=51, y=34$ and $x=17$. Hence $H_{n}$ has perimeter $204$. So, $H_{n-1}$ must have area $204$, which by a similar argument leads to the triangle $z=17, y=9$, and $x=8$ with perimeter $68$. This implies $H_{n-2}$ has area $68$, but there is no such Heronian triangle. Since this process stops, we cannot have a $n$-\emph{sociable} cycle that contains the triangle $(3, 865, 866)$. 
 
 The only triangle left is the triangle with side lengths $(3,25,26)$. This triangle has area $36$ and perimeter $54$. Similar computations show that there is only one triangle with area $54$, namely, the triangle $(9, 12, 15)$. This is the amicable pair found in \cite{Praton}. Therefore the only triangle that precedes $(3, 25, 26)$ is the triangle $(9, 12, 15)$ that has perimeter $36$. The same kind of computations show that the only triangles with area $36$ are $(3, 25, 26)$ and $(9, 10, 17)$. Thus, the only possibilities of the predecessor of $(9, 12, 15)$ or $(9, 10, 17)$ are $(9,10,17)$ or the amicable pair. By Lemma \ref{lm1}, all $n$\emph{-sociable} Heronian cycles must contain the triangle $(3, 25, 26)$ since this is the only sociable triangle whose perimeter is greater than its area. This concludes the fact that each $n$\emph{-sociable} cycle is only composed of amicable pairs and the triangle $(9, 10, 17)$.
 \end{proof}
 
 \section*{Conclusion}
 
 Let $v$ represent the triangle $(3, 25, 26)$, $u$ represent the triangle $(9, 12, 15)$, and $w$ represent the triangle $(9, 10, 17)$. Theorem \ref{thm3} states that any Heronian cycle is only composed of $u,v$ and $w$. All cycles must contain $v$, which must be preceded by $u$. Hence, $u$ and $v$ come in pairs. Additionally, a pair $u, v$ is proceeded or succeeded by the same pair or by $w$. We immediately construct a Heronian cycle as follows: 
 \begin{align}\label{ff}
   \underbrace{[u, v, w, w, \dots, w]}_{n},  
 \end{align} 
 where the $i$th component denotes the $i$th triangle. Note that (\ref{ff}) represents all cycles that have exactly one pair $u ,v $ since Heronian cycles are equivalent with respect to rotation; for example, $[u,v,w,w,w]$, $[w, w, u,v, w]$ and $[w, w, w, u, v]$ all represent the same cycle.
\begin{figure}[hbt]
    \centering
    \includegraphics[width=4.5in]{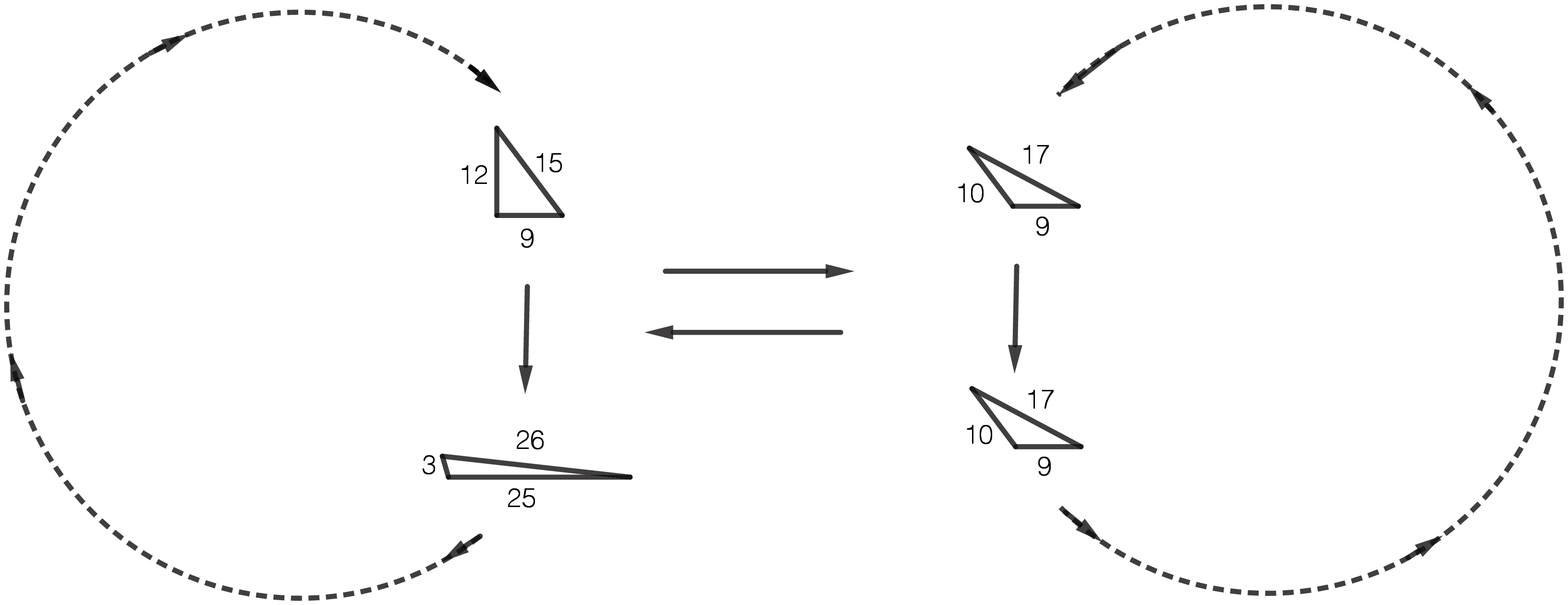}
    \caption{We can replace each \emph{amicable} pair with a pair of equable triangles $(9, 10, 17)$. This is because the preceding triangle of the pair $(9, 12, 15) \to (3, 25, 26)$ must have perimeter $36$ and similarly the successor triangle must have area $36$}
    \label{eq}
\end{figure}
Next, we can replace a pair of equable triangles with our amicable pair (see figure \ref{eq}). Thus, we can form an additional $(n-2)/2$ (if $n$ is even) and $(n-3)/2$ (if $n$ is odd) Heronian cycles by successive replacements in each step,
\begin{align*}
    \underbrace{[u, v, u, v, \overbrace{w, \dots, w}^{n-4}]}_{n}, \underbrace{[u, v, u, v, u, v, \overbrace{w \dots w}^{n-6}]}_{n}, \dots
\end{align*}
 Starting with our cycle in (\ref{ff}) and applying this process we get $\lfloor n/2 \rfloor$ $n$\emph{-sociable} distinct Heronian cycles for each $n$. However, this list does not contain all $n$\emph{-cycles} as for instance $[u, v, u, v, w, w]$ is a different cycle from $[u, v, w, u, v, w]$. Nevertheless, we can use this process to identify all $n$\emph{-cycles} for each $n$ by considering all the possibilities that arise from having $k$ pairs $u, v$ and $n-2k$ $w$s for $1 \leq k \leq \lfloor n/2 \rfloor$. For example, there are only two $5$-\emph{sociable} Heronian cycles: $[u, v, u, v, w]$ and $[u, v, w, w ,w]$, as shown in figure \ref{5sociable} .

\begin{figure}[hbt]
    \centering
    \includegraphics[width=4.5in]{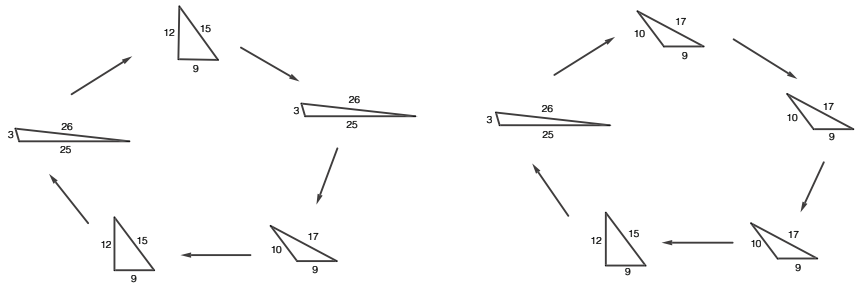}
    \caption{All possible $5$-\emph{sociable} cycles.}
    \label{5sociable}
\end{figure}

One can claim that there essentially are no non-trivial $n$-\emph{sociable} cycles except for $n=2$. All of the other cycles are formed from \emph{amicable} pairs and the equable triangle with area and perimeter $36$. Surprisingly, the complete characterization of $n$-\emph{sociable} cycles turns out to be provable by only using classical number theory and elementary methods as opposed to some more complex methods such as those involving algebraic number theory used in \cite{Hirakawa} for other questions on Heronian triangles.

\section*{About the author:}
   Nart Shalqini is an international student from Kosovo; a young and small country in the Balkans. He was born in Istanbul in 1999 and it is his favorite place to visit. Besides Mathematics, Nart enjoys playing board games and reading mystery novels in his spare time.

\subsection*{Primus Scriber}
   Franklin and Marshall College \\
   Lancaster, Pennsylvania, 17603 \\
   nshalqin@fandm.edu

\end{document}